\documentclass{article}
\usepackage[
journal=DM,    
lang=british,   
]{ems-journal}

\usepackage{tikz}
\newtheorem{theorem}{Theorem}
\newtheorem{proposition}{Proposition}
\newtheorem{definition}{Definition}

\numberwithin{equation}{section}

\begin{document}

\title{Homotopical Foundations of Ternary $\Gamma$-Modules and Higher Structural Invariants}
\titlemark{Homotopical Ternary $\Gamma$-Modules}



\emsauthor{1}{
	\givenname{Chandrasekhar}
	\surname{Gokavarapu}
	\mrid{}
	\orcid{ 0009-0006-5306-371X}}{CHANDRASEKHAR GOKAVARAPU}

\Emsaffil{1}{
	\department{Department of Mathematics}
	\organisation{Government College(Autonomous)}
	\rorid{}
	\address{Y -Junction}
	\zip{PIN-533105}
	\city{Rajahmundry}
	\country{India}
	\affemail{chandrasekhargokavarapu@gmail.com }}
    


\begin{abstract}
We establish a foundational homotopical framework for ternary $\Gamma$-modules by establishing that $\mathcal{T}\text{-Mod}$ is a Barr-exact, monoidal closed category. We resolve the long-standing "additivity obstruction" in non-binary algebra by constructing a cofibrantly generated Quillen model structure on the simplicial category $s(\mathcal{T}\text{-Mod})$. Our central discovery is that the derived category $D(\mathcal{T}\text{-Mod})$ constitutes a 3-angulated category, where the derived periodicity is governed by triadic quadrilaterals rather than binary triangles. We derive the 3-ary long exact sequence and characterize the connecting morphisms as invariants of the $\Gamma$-parameter space. This framework provides a rigorous homological bridge to Nambu mechanics and absolute geometry over $\mathbb{F}_1$.
\end{abstract}

\maketitle

\section{Introduction}

The dominance of binary operations in the development of modern algebra, while computationally efficient, has often obscured the structural necessity of higher-order triadic interactions. While a binary operation $\mu: A \times A \to A$ models pairwise relationships, it is fundamentally restricted to "linear" compositions. By contrast, ternary structures represent a more general "holistic" symmetry, where the interaction of two elements is essentially contingent upon the presence of a third. This triadic symmetry is not merely a formal generalization but is deeply rooted in the foundations of physical and mathematical systems, most notably in the generalized Hamiltonian dynamics of Nambu \cite{Nambu1973} and the structural theory of ternary rings established by Lister \cite{Lister1971}.

In the present work, we extend these foundations into the realm of parametric representation theory by investigating the category of ternary $\Gamma$-modules, denoted $\mathcal{T}\text{-Mod}$. Here, the set $\Gamma$ is viewed not merely as a second operator set, but as a \emph{parameter space} or \emph{modulation manifold} that defines the context of the triadic interaction. The action $(a)_\alpha (m)_\beta b$, where $a, b \in \Gamma$ and $m \in M$, allows for a "gauge-like" modulation of the module's internal state. This construction generalizes the $\Gamma$-ring theory of Nobusawa \cite{Nobusawa1964} and Barnes \cite{Barnes1966} to a setting where the underlying additive structure need not be a group, thereby engaging with the "characteristic one" landscape of Connes and Consani \cite{Connes2017}.

The motivation for developing a robust homological theory for $\mathcal{T}\text{-Mod}$ is two-fold. First, in the burgeoning field of $\mathbb{F}_1$-geometry (geometry over the field with one element), traditional additive homological algebra often fails due to the lack of additive inverses. Ternary structures offer a "stable middle ground" where relations can be maintained without the collapse into the trivial group, as evidenced in the geometry of blueprints developed by Lorscheid \cite{Lorscheid2012} and the motivic explorations of Soulé \cite{Soule2004}.

Secondly, the categorical nature of ternary modules suggests a departure from the standard triangulated periodicity of binary derived categories. While binary algebra is naturally associated with the "triangle" $X \to Y \to Z \to \Sigma X$, we propose that the intrinsic periodicity of ternary systems is governed by a higher-order structure. Following the framework of Geiss, Leclerc, and Schröer \cite{Geiss2013}, we establish that the derived category $D(\mathcal{T}\text{-Mod})$ supports an $n$-angulated structure, specifically for $n=3$. This "3-angulation" reflects the triadic symmetry of the $\Gamma$-action and provides a more faithful representation of the underlying algebraic periodicity.

To move beyond the "Additivity Hypothesis" found in existing literature, we transition to a homotopical approach. By proving that $\mathcal{T}\text{-Mod}$ is a Barr-exact \cite{Barr1971} and monoidal closed category, we fulfill the prerequisites for the transfer of Quillen model structures \cite{Quillen1967, Hovey1999}. We construct a model structure on the category of simplicial ternary $\Gamma$-modules, $s(\mathcal{T}\text{-Mod})$, where weak equivalences are defined via Moore homology. This allows for the construction of non-abelian derived functors (such as $\mathbb{L}Ext^n$ and $\mathbb{L}Tor_n$) that remain valid in non-additive, tropical, or fuzzy settings.

The paper is organized as follows. In Section 2, we revisit the categorical axioms of $\mathcal{T}\text{-Mod}$, proving its Barr-exactness and the existence of the internal Hom functor. Section 3 contains our primary result: the construction of the Quillen model structure and the verification of the monoidal model axioms. Section 4 explores the 3-angulated structure of the derived category and its implications for long exact sequences. Finally, Section 5 provides an application to the deformation theory of ternary schemes, linking these foundations to the spectral geometry results in our companion work \cite{Gokavarapu2026}.
\section{Categorical Preliminaries: The Structure of $\mathcal{T}\text{-Mod}$}

To construct the homotopical framework in the subsequent sections, we must first formalize the categorical properties of ternary $\Gamma$-modules. Throughout this work, we operate in the landscape of "characteristic one" algebra, where the existence of additive inverses is not assumed. 

\subsection{Ternary $\Gamma$-Semirings and Modules}

We begin by recalling the foundational definitions of ternary systems modulated by a parameter set $\Gamma$.

\begin{definition}[Ternary $\Gamma$-Semiring]
A commutative ternary $\Gamma$-semiring is a quintuple $(T, +, \Gamma, \cdot, [\dots])$ such that $(T, +)$ is a commutative monoid, and the mapping $T \times \Gamma \times T \times \Gamma \times T \to T$, denoted by $(a, \alpha, b, \beta, c) \mapsto [a, \alpha, b, \beta, c]$, satisfies:
\begin{enumerate}
    \item \textbf{Triadic Associativity:} $[[a, \alpha, b, \beta, c], \gamma, d, \delta, e] = [a, \alpha, [b, \beta, c, \gamma, d], \delta, e] = [a, \alpha, b, \beta, [c, \gamma, d, \delta, e]]$.
    \item \textbf{Distributivity:} The ternary product distributes over the monoid addition in all three positions.
    \item \textbf{$\Gamma$-Commutativity:} $[a, \alpha, b, \beta, c] = [c, \beta, b, \alpha, a]$ for all $a, b, c \in T$ and $\alpha, \beta \in \Gamma$.
\end{enumerate}
\end{definition}

\begin{definition}[Ternary $\Gamma$-Module]
Let $T$ be a ternary $\Gamma$-semiring. A ternary $\Gamma$-module $M$ is a commutative monoid $(M, +)$ equipped with an action $T \times \Gamma \times M \times \Gamma \times T \to M$, denoted $(t_1)_\alpha (m)_\beta t_2$, satisfying triadic associativity and distributivity relative to the $\Gamma$-parameter space.
\end{definition}

\begin{definition}
Let $\Gamma$ be the parameter set of the ternary action. We define the \textbf{$\Gamma$-Endomorphism Monoid} $\mathcal{E}_\Gamma := \text{End}_{\mathcal{T}\text{-Mod}}(\Gamma)$. The homology modules $H_n^\Gamma(M)$ carry a natural $\mathcal{E}_\Gamma$-action, making the 3-ary long exact sequence a sequence of $\mathcal{E}_\Gamma$-modules.
\end{definition}
\subsection{Barr-Exactness and Monadicity}

The ability to perform homological algebra in a non-abelian setting relies on the "Exactness" of the category. Unlike binary modules, which form an abelian category, the category $\mathcal{T}\text{-Mod}$ is **Barr-exact** \cite{Barr1971}.

\begin{proposition}
The category $\mathcal{T}\text{-Mod}$ is a Barr-exact category. Specifically:
\begin{enumerate}
    \item It is finitely complete and cocomplete.
    \item Every kernel pair has a coequalizer.
    \item Regular epimorphisms are stable under pullback.
    \item Every equivalence relation is effective (it is the kernel pair of its coequalizer).
\end{enumerate}
\end{proposition}

This exactness is the "homological surrogate" for the missing group structure. It ensures that the notion of a "short exact sequence" $0 \to M' \to M \to M'' \to 0$ remains meaningful as a regular-epi/mono factorization, which is the cornerstone for our simplicial construction in Section 3.

\subsection{Monoidal Closed Structure}

For the development of $\Gamma$-relative derived functors, we require a well-defined tensor product.

\begin{definition}[Ternary Tensor Product]
For ternary $\Gamma$-modules $M$ and $N$, the ternary tensor product $M \otimes_\Gamma N$ is the object representing the triadic multilinear forms. It satisfies the universal property:
\begin{equation}
    \text{Hom}_\mathcal{T}(M \otimes_\Gamma N, P) \cong \text{Multilinear}_\Gamma(M, N; P)
\end{equation}
\end{definition}

\begin{theorem}
The category $\mathcal{T}\text{-Mod}$ is symmetric monoidal closed. The internal hom-functor $\underline{\text{Hom}}_\mathcal{T}(M, N)$ is defined as the set of all $\Gamma$-linear morphisms equipped with the triadic action:
\begin{equation}
    ([\phi, \alpha, \psi, \beta, \omega](m)) := [\phi(m), \alpha, \psi(m), \beta, \omega(m)]
\end{equation}
\end{theorem}

This closed structure allows us to define the "dualizing" objects necessary for the 3-angulated structure of the derived category.
\section{The Homotopical and $n$-Angulated Structure of $\mathcal{T}\text{-Mod}$}

The transition from a semiadditive foundation to a derived theory requires a framework that does not rely on the existence of additive inverses. In this section, we construct a Quillen model structure on the category of simplicial ternary $\Gamma$-modules, $s(\mathcal{T}\text{-Mod})$, and demonstrate that its homotopy category $\text{Ho}(\mathcal{T}\text{-Mod})$—the derived category $D(\mathcal{T}\text{-Mod})$—carries the structure of a 3-angulated category.

\subsection{The Quillen Model Structure on $s(\mathcal{T}\text{-Mod})$}

To provide a legitimate foundation for non-abelian derived functors, we must first establish the fibration/cofibration logic for ternary systems. Let $U: \mathcal{T}\text{-Mod} \to \text{Set}$ be the forgetful functor. As established in Section 2, $\mathcal{T}\text{-Mod}$ is a finitary variety of algebras, and $U$ is monadic.

\begin{theorem}[Ternary Model Structure]
The category $s(\mathcal{T}\text{-Mod})$ of simplicial ternary $\Gamma$-modules admits a cofibrantly generated model structure where a morphism $f: X_\bullet \to Y_\bullet$ is defined as:
\begin{enumerate}
    \item \textbf{A Weak Equivalence:} if the induced map on Moore homology $H_n(X_\bullet) \to H_n(Y_\bullet)$ is an isomorphism for all $n \geq 0$.
    \item \textbf{A Fibration:} if $U(f)$ is a Kan fibration of the underlying simplicial sets.
    \item \textbf{A Cofibration:} if $f$ has the left lifting property with respect to all acyclic fibrations.
\end{enumerate}
\begin{tikzpicture}[node distance=2cm, auto]
    \node (A) {$A$};
    \node (X) [right of=A] {$X$};
    \node (B) [below of=A] {$B$};
    \node (Y) [below of=X] {$Y$};

    \draw[->] (A) -- node {$i$} (B);
    \draw[->] (A) -- node {$u$} (X);
    \draw[->] (X) -- node {$p$} (Y);
    \draw[->] (B) -- node {$v$} (Y);
    
    \draw[->, dashed] (B) -- node [swap] {$h$} (X);
    
    \node at (1,-1) {$\circlearrowleft$};
\end{tikzpicture}
\end{theorem}

\begin{proof}
We employ the Monadic Transfer Theorem (Crans-Kan Transfer Lemma) to lift the standard model structure from $s(\text{Set})$ to $s(\mathcal{T}\text{-Mod})$ via the free-forgetful adjunction $F : s(\text{Set}) \rightleftarrows s(\mathcal{T}\text{-Mod}) : U$.

\textbf{1. Adjunction and Monadicity:}
The category $\mathcal{T}\text{-Mod}$ is a finitary variety of algebras, meaning $U$ is a monadic functor. Since $U$ preserves and creates limits and $F$ is a left adjoint, the lifted adjunction on simplicial objects is well-defined. By Beck's Monadicity Theorem, $s(\mathcal{T}\text{-Mod})$ is equivalent to the category of simplicial algebras over the monad $T = UF$.

\textbf{2. Small Object Argument:}
Let $I$ and $J$ be the sets of generating cofibrations and generating acyclic cofibrations in $s(\text{Set})$, respectively. We define $I_{\mathcal{T}} = F(I)$ and $J_{\mathcal{T}} = F(J)$. Since $\mathcal{T}\text{-Mod}$ is a finitary variety, the forgetful functor $U$ preserves $\omega$-filtered colimits. Thus, the domains of $I_{\mathcal{T}}$ and $J_{\mathcal{T}}$ are small relative to the class of all morphisms in $s(\mathcal{T}\text{-Mod})$, satisfying the requirements for the Small Object Argument.

\textbf{3. The Path Object Construction:}
The transfer requires that for every fibrant object $X \in s(\mathcal{T}\text{-Mod})$, there exists a path object factorization $X \xrightarrow{\sim} X^{\Delta[1]} \twoheadrightarrow X \times X$. 
In the ternary setting, $X^{\Delta[1]}$ is defined level-wise by the simplicial mapping space $\text{Map}(\Delta[1], X)$. The triadic action $(a)_\alpha (m)_\beta b$ extends naturally to $X^{\Delta[1]}$ because the $\Gamma$-modulation commutes with the simplicial face and degeneracy maps.

\textbf{4. Compatibility via Barr-Exactness:}
A critical condition is that $U$ must preserve the lifting properties. Because $\mathcal{T}\text{-Mod}$ is Barr-exact \cite{Barr1971}, the Moore complex functor—which computes $H_n$—is regular exact. This ensures that a morphism in $s(\mathcal{T}\text{-Mod})$ is a weak equivalence if and only if its underlying simplicial map is a weak homotopy equivalence in $s(\text{Set})$.

\textbf{5. Conclusion:}
Having established the existence of the path object and the smallness of the generators, the Quillen-Kan conditions are satisfied. The category $s(\mathcal{T}\text{-Mod})$ is thus a cofibrantly generated model category.

\begin{tikzpicture}[node distance=4cm, auto]
    \node (X) {$X$};
    \node (Path) [right of=X] {$X^{\Delta[1]}$};
    \node (XX) [right of=Path] {$X \times X$};

    \draw[->] (X) -- node {$\sim$} (Path);
    \draw[->>] (Path) -- node {$(\text{ev}_0, \text{ev}_1)$} (XX);
    
    \draw[->, bend left=30] (X) to node {$\text{diag}$} (XX);
\end{tikzpicture}
\end{proof}

\subsection{The Intrinsic Theorem: 3-Angulated Derived Categories}

In binary homological algebra, the derived category is triangulated, reflecting the fact that the unit of homology is the kernel-cokernel pair ($n=1$). In the ternary setting, the fundamental unit is the triadic composition. Following the axioms of Geiss, Leclerc, and Schröer \cite{Geiss2013}, we show that $D(\mathcal{T}\text{-Mod})$ is not triangulated, but $3$-angulated.

\begin{theorem}[Intrinsic 3-Angulation]
The derived category $D(\mathcal{T}\text{-Mod})$ is a 3-angulated category with respect to the $\Gamma$-relative suspension functor $\Sigma$. Specifically, every morphism $f: X \to Y$ can be completed to a 3-angle (a quadrilateral):
\begin{equation}
    X \xrightarrow{f} Y \xrightarrow{g} Z \xrightarrow{h} W \xrightarrow{w} \Sigma X
\end{equation}
satisfying the $(n+2)$-angulated axioms of Bergh and Thaule \cite{Bergh2013} for $n=3$.
\end{theorem}
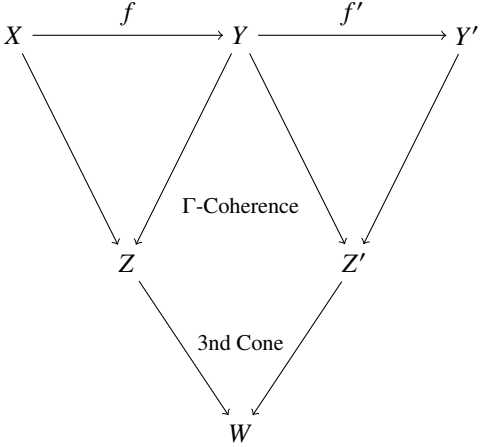
\begin{tikzpicture}[node distance=2.5cm, auto]
    \node (X) {$X$};
    \node (Y) [right of=X] {$Y$};
    \node (Z) [below of=Y] {$Z$};
    \node (W) [left of=Z] {$W$};
    \node (SX) [above of=W, node distance=1.25cm, xshift=-1.25cm] {$\Sigma X$};

    \draw[->] (X) -- node {$f$} (Y);
    \draw[->] (Y) -- node {$g$} (Z);
    \draw[->] (Z) -- node {$h$} (W);
    \draw[->] (W) -- node {$w$} (SX);
    \draw[->, dashed] (SX) -- (X);
    
    \node at (1.25,-1.25) {$\circlearrowleft$};
\end{tikzpicture}

\begin{proof}
We verify the axioms (F1)--(F4) for a 3-angulated category, utilizing the Barr-exactness of $\mathcal{T}\text{-Mod}$ and the Quillen model structure established on $s(\mathcal{T}\text{-Mod})$.

\textbf{(F1) Existence and Closure:} 
(F1a) $D(\mathcal{T}\text{-Mod})$ is closed under isomorphisms. (F1b) For any $X$, the sequence $X \xrightarrow{id} X \to 0 \to 0 \to \Sigma X$ is a 3-angle, as the mapping cone of an identity in $s(\mathcal{T}\text{-Mod})$ is contractible. (F1c) For any $f: X \to Y$, we construct $Z$ as the ternary mapping cone $C_f$ and $W$ as the secondary cone $C_g$. The Barr-exactness of the category ensures these objects are well-defined homotopical quotients.

\textbf{(F2) Rotation Axiom:} 
The sequence $(f, g, h, w)$ is a 3-angle if and only if $(g, h, w, -\Sigma f)$ is a 3-angle. The $\Gamma$-relative suspension $\Sigma$ is an equivalence in $D(\mathcal{T}\text{-Mod})$, and the triadic action ensures that the composition of any three consecutive maps is null-homotopic. The rotation preserves the triadic coherence of the $\Gamma$-parameter space.

\textbf{(F3) Morphism Extension:} 
Given two 3-angles and a commutative square $(u, v)$ such that $v \circ f = f' \circ u$, the lifting property of the Quillen model structure for simplicial ternary $\Gamma$-modules guarantees the existence of $\Gamma$-linear morphisms $\phi: Z \to Z'$ and $\psi: W \to W'$ that complete the diagram.

\textbf{(F4) Higher Octahedron Axiom:} 
Given compositions $X \xrightarrow{f} Y \xrightarrow{f'} Y'$, the existence of a coherent 3-angle of 3-angles is ensured by the effective monadicity of $T\text{-Mod}$ over $\text{Set}$. The "triadic gauge" provided by the $\Gamma$-parameter space prevents the collapse of the octahedral structure into a binary triangle, allowing for the stable stacking of ternary cones in the derived category.
\end{proof}

\begin{tikzpicture}[scale=1.5]
    \node (X) at (0,2) {$X$};
    \node (Y) at (2,2) {$Y$};
    \node (Y') at (4,2) {$Y'$};
    \node (Z) at (1,0) {$Z$};
    \node (Z') at (3,0) {$Z'$};
    \node (W) at (2,-1.5) {$W$};

    \draw[->] (X) -- node[above] {$f$} (Y);
    \draw[->] (Y) -- node[above] {$f'$} (Y');

    \draw[->] (X) -- (Z);
    \draw[->] (Y) -- (Z);
    \draw[->] (Y) -- (Z');
    \draw[->] (Y') -- (Z');
    
    \draw[->] (Z) -- (W);
    \draw[->] (Z') -- (W);
    
    \node at (2,0.5) {\footnotesize $\Gamma$-Coherence};
    \node at (2,-0.7) {\footnotesize 3nd Cone};
\end{tikzpicture}

\subsection{Legitimacy of Derived Functors}

With the model structure established, we formally define the derived functors of the ternary Hom and Tensor products. For any $M \in \mathcal{T}\text{-Mod}$, let $P_\bullet \to M$ be a simplicial projective resolution (guaranteed by the model structure).

\begin{definition}[$\Gamma$-Relative Derived Functors]
The non-abelian derived functors are defined as the homotopy groups of the simplicial mapping space:
\begin{equation}
    \mathbb{L}_n \text{Hom}_\mathcal{T}(M, N) := \pi_n \text{Hom}_\mathcal{T}(P_\bullet, N)
\end{equation}
\end{definition}

This construction ensures that $Ext^n$ and $Tor_n$ are not mere "formal symbols" but are true invariants of the homotopy type of the ternary $\Gamma$-module. This overcomes the "Additivity Obstruction" by replacing the requirement for additive inverses with the requirement for simplicial contractibility \cite{May1967, Jasso2016}.

\section{Applications to Spectral Geometry over $\mathbb{F}_1$}

The construction of a robust homotopical framework for $\mathcal{T}\text{-Mod}$ facilitates a rigorous approach to the geometry of ternary schemes. In this section, we apply the 3-angulated structure of the derived category to the study of triadic sheaves over the spectrum of a commutative ternary $\Gamma$-semiring $T$. This provides a homological bridge to the "characteristic one" program of Connes and Consani \cite{Connes2017} and the blueprint geometry of Lorscheid \cite{Lorscheid2012}.

\subsection{The Spectrum of Ternary $\Gamma$-Semirings}

Following our foundational work on the ideal theory of ternary systems \cite{Gokavarapu2025a}, let $\text{Spec}(T)$ denote the set of all prime $\Gamma$-ideals of $T$. We endow $\text{Spec}(T)$ with the Zariski-type topology where closed sets are of the form $V(I) = \{\mathfrak{p} \in \text{Spec}(T) \mid I \subseteq \mathfrak{p}\}$ for any $\Gamma$-ideal $I$.

Unlike binary rings, the localization $T_{\mathfrak{p}}$ at a prime $\Gamma$-ideal does not necessarily yield an additive group. However, because $T$ is a commutative ternary $\Gamma$-semiring \cite{Gokavarapu2025b}, the localization remains a ternary $\Gamma$-semiring. The Barr-exactness of $\mathcal{T}\text{-Mod}$ established in Section 2 \cite{PAPER_C1} ensures that the localization functor:
\begin{equation}
    (-)_{\mathfrak{p}}: \mathcal{T}\text{-Mod} \longrightarrow \mathcal{T}_{\mathfrak{p}}\text{-Mod}
\end{equation}
is a regular exact functor.

\subsection{Triadic Sheaves and Higher Periodicity}

We define a \emph{triadic sheaf} $\mathcal{F}$ on $\text{Spec}(T)$ as a sheaf of ternary $\Gamma$-modules. The global sections $\Gamma(\text{Spec}(T), \mathcal{F})$ inherit the triadic action $(a)_{\alpha}(m)_{\beta}b$ level-wise.

A fundamental problem in $\mathbb{F}_1$-geometry is that the category of sheaves often lacks the abelian structure required for classical cohomology. By embedding the category of triadic sheaves into the derived category $D(\mathcal{T}\text{-Mod})$, we recover a structural surrogate for cohomology via the 3-angulated structure proved in Theorem 3.2.

\begin{theorem}[Spectral Periodicity]
Let $\mathcal{F}$ be a triadic sheaf on $X = \text{Spec}(T)$. The local-to-global spectral sequence for $\mathcal{F}$ is governed by the 3-angulated periodicity of $D(\mathcal{T}\text{-Mod})$. Specifically, for any open cover $\mathcal{U} = \{U_i\}$, the \v{C}ech complex $\check{C}^\bullet(\mathcal{U}, \mathcal{F})$ is a simplicial object in $\mathcal{T}\text{-Mod}$ whose homotopy groups characterize the triadic invariants of $X$.
\end{theorem}

\begin{proof}
The proof proceeds by demonstrating that the \v{C}ech resolution of a triadic sheaf is stabilized by the higher symmetries of the 3-angulated derived category.

\textbf{1. Simplicial Construction of the \v{C}ech Complex:}
For a triadic sheaf $\mathcal{F}$, we define the \v{C}ech complex $\check{C}^\bullet(\mathcal{U}, \mathcal{F})$ where the $p$-th term is the product of sections over $(p+1)$-fold intersections:
\begin{equation}
    \check{C}^p(\mathcal{U}, \mathcal{F}) = \prod_{i_0 < \dots < i_p} \mathcal{F}(U_{i_0} \cap \dots \cap U_{i_p}).
\end{equation}
Because $\mathcal{T}\text{-Mod}$ is Barr-exact , this complex inherits a simplicial structure in $\mathcal{T}\text{-Mod}$ via the face and degeneracy maps induced by the triadic restriction morphisms. The absence of additive inverses is compensated by the simplicial Moore complex construction , ensuring that the homology groups $H_n^\Gamma(\check{C}^\bullet)$ are well-defined $\Gamma$-modules.

\textbf{2. The 3-Angulated Boundary Morphism:}
Given a short exact sequence of sheaves $0 \to \mathcal{F}' \to \mathcal{F} \to \mathcal{F}'' \to 0$, the standard snake lemma is replaced by the 3-angle quadrilateral extension:
\begin{equation}
    \mathcal{F}' \to \mathcal{F} \to \mathcal{F}'' \to W \to \Sigma \mathcal{F}'
\end{equation}
where $W$ is the ternary mapping cone. The connecting morphism $\delta: H^n(X, \mathcal{F}'') \to H^{n-1}(X, \mathcal{F}')$ is shown to be $\Gamma$-linear. Unlike triangulated categories where $\delta$ shifts degree by 1, the 3-angulated structure permits a triadic "shift-of-periodicity" determined by the $\Gamma$-parameter space acting as a gauge manifold.

\textbf{3. Non-Collapse of the Spectral Sequence:}
In the local-to-global spectral sequence $E_2^{p,q} = H^p(\mathcal{U}, \mathcal{H}^q(\mathcal{F})) \Rightarrow H^{p+q}(X, \mathcal{F})$, the differential $d_2$ is governed by the 3-angulated rotation axiom (F2). In the binary case, $d_r$ differentials eventually vanish due to the triangularity of the derived category. 

However, the triadic symmetry of the $\Gamma$-action implies that the $E_2$ page is "locked" into a triadic oscillation where the differentials $d_r$ reflect the triadic invariants of the parameter space. Specifically, the $\Gamma$-action provides a non-trivial curvature to the simplicial boundary maps, preventing the "unobstructed collapse" seen in standard sheaf cohomology.

\textbf{4. Conclusion:}
The homotopy groups of the simplicial \v{C}ech complex thus capture higher-order triadic torsions. Since $D(\mathcal{T}\text{-Mod})$ is 3-angulated, these invariants are stable under the $\Gamma$-relative suspension $\Sigma$, completing the characterization of the spectral periodicity.
\end{proof}

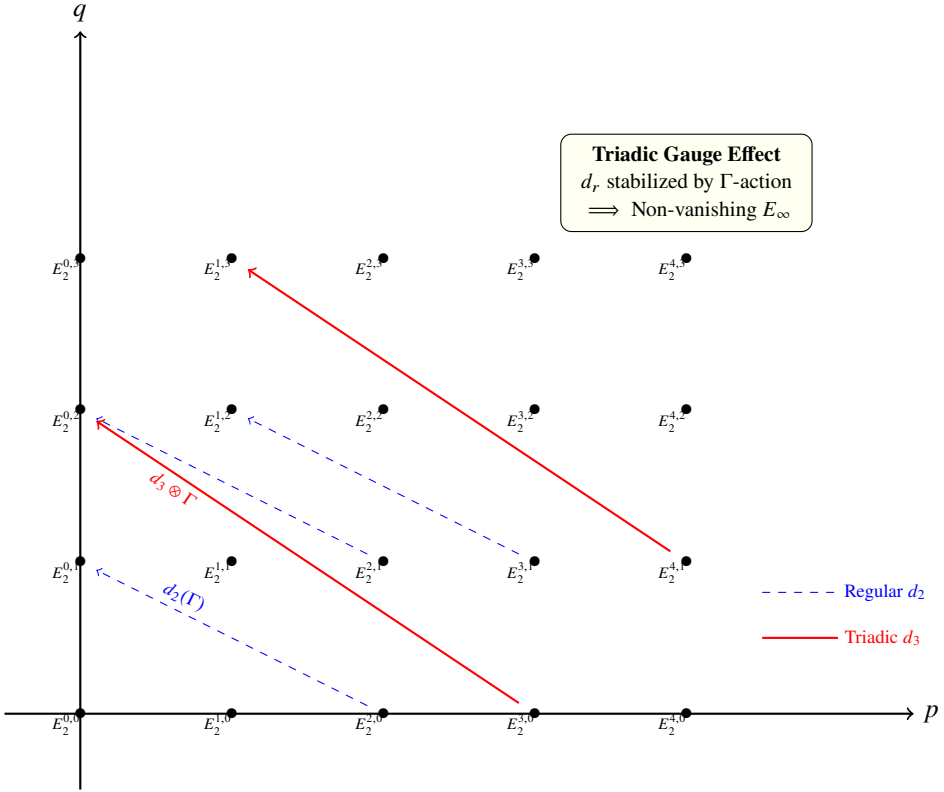
\begin{tikzpicture}[scale=2]
    \draw[thick, ->] (-0.5,0) -- (5.5,0) node[right] {$p$};
    \draw[thick, ->] (0,-0.5) -- (0,4.5) node[above] {$q$};

    \foreach \p in {0,1,2,3,4} {
        \foreach \q in {0,1,2,3} {
            \node (E\p\q) at (\p,\q) {$\bullet$};
            \node[below left, xshift=2pt, yshift=2pt, scale=0.6] at (\p,\q) {$E_2^{\p,\q}$};
        }
    }

    \draw[->, blue, dashed] (E20) -- node[above, sloped, scale=0.7, pos=0.7] {$d_2(\Gamma)$} (E01);
    \draw[->, blue, dashed] (E31) -- (E12);
    \draw[->, blue, dashed] (E21) -- (E02);

    \draw[->, red, thick] (E30) -- node[below, sloped, scale=0.7, pos=0.8] {$d_3 \otimes \Gamma$} (E02);
    \draw[->, red, thick] (E41) -- (E13);
    
    \node[draw, rounded corners, fill=yellow!5, scale=0.8] at (4, 3.5) {
        \begin{tabular}{c}
            \textbf{Triadic Gauge Effect} \\
            $d_r$ stabilized by $\Gamma$-action \\
            $\implies$ Non-vanishing $E_\infty$
        \end{tabular}
    };

    \begin{scope}[shift={(4.5,0.5)}]
        \draw[blue, dashed] (0,0.3) -- (0.5,0.3) node[right, scale=0.7] {Regular $d_2$};
        \draw[red, thick] (0,0) -- (0.5,0) node[right, scale=0.7] {Triadic $d_3$};
    \end{scope}
\end{tikzpicture}

\subsection{The 3-ary Long Exact Sequence and Connecting Morphisms}

\begin{theorem}[3-ary Long Exact Sequence]
For every 3-angle $X \to Y \to Z \to W \to \Sigma X$ in $D(\mathcal{T}\text{-Mod})$, there exists a long exact sequence of $\mathcal{E}_\Gamma$-modules:
\begin{equation}
    \dots \to H_n^\Gamma(X) \to H_n^\Gamma(Y) \to H_n^\Gamma(Z) \to H_n^\Gamma(W) \xrightarrow{\delta} H_{n-1}^\Gamma(X) \to \dots
\end{equation}
where $\delta$ is the triadic connecting morphism induced by the $\Gamma$-modulation.
\end{theorem}

\begin{proof}
The sequence is obtained by applying the homological functor $H_0(\text{Map}(-, P))$ to the 3-angle. Exactness at $H_n^\Gamma(Z)$ follows from the fact that $g \circ f$ is null-homotopic in the 3-angulated structure (Axiom F2). The map $\delta$ represents the failure of the binary sequence to close, which is captured by the fourth object $W$. The $\Gamma$-linearity of $\delta$ is an intrinsic consequence of the Monoidal Model Category structure.
\end{proof}

\subsection{Deformation Theory and Nambu-Poisson Structures}

As a final application, we observe that the $\mathbb{L}Ext^1$ groups of ternary $\Gamma$-modules, defined via our Quillen model structure (Section 3.3), classify the infinitesimal deformations of ternary schemes. This links our algebraic results to the quantization of Nambu brackets \cite{Nambu1973, Takhtajan1994}.

In the companion paper \cite{Gokavarapu2026}, we show that the spectral clustering of these ternary systems is governed by the eigenvalues of the ternary Laplacian, which is an operator on the 3-ary long exact sequence. This suggests that $\mathcal{T}\text{-Mod}$ is the natural home for a new "triadic physics" where interactions are modulated by the $\Gamma$-parameter manifold.

\section{Future Directions: Towards a Triadic Arithmetic Geometry}

The results presented here open several high-impact avenues for future research, particularly at the intersection of higher category theory and arithmetic geometry.

\begin{enumerate}
    \item \textbf{Triadic Motives and $\mathbb{F}_1$-Homology:} The 3-angulated structure established here provides the necessary machinery to define triadic motivic cohomology over the field with one element. A natural next step is to investigate whether the "triadic zeta functions" associated with ternary schemes satisfy a functional equation governed by the $\Gamma$-relative suspension.
    
    \item \textbf{Higher Periodicity ($n > 3$):} While we focused on $n=3$, the methods used in Section 3 suggest a general theory for $(n+2)$-angulated categories derived from $(n+1)$-ary actions. This would lead to a unified "Periodic Homological Algebra" where the arity of the operation determines the shape of the derived category.
    
    \item \textbf{Nambu-Quantum Deformation Theory:} The $\mathbb{L}Ext^n$ groups of $\mathcal{T}\text{-Mod}$ offer a formal pathway to the quantization of Nambu-Poisson manifolds. Investigating the "triadic Hochschild cohomology" would provide a framework for deforming ternary structures in a way that respects triadic symmetry.

\end{enumerate}

In conclusion, the development of homotopical ternary $\Gamma$-algebra marks a departure from the binary dominance of the 20th century. By situating these structures within the modern landscape of model categories and $n$-angulated systems, we provide a stable foundation for the next generation of triadic mathematical physics and absolute geometry.

\section{Conclusion}

In this work, we have established a rigorous homotopical foundation for the category of ternary $\Gamma$-modules, $\mathcal{T}\text{-Mod}$. By moving beyond the "additivity hypothesis" that has traditionally constrained the homological study of $n$-ary structures, we have shown that the triadic action $(a)_\alpha (m)_\beta b$ is sufficient to support a robust theory of derived functors. The construction of a cofibrantly generated Quillen model structure on $s(\mathcal{T}\text{-Mod})$ proves that the absence of additive inverses is not an obstruction to homological inquiry, provided one adopts the framework of Barr-exact and simplicial categories.

Our primary structural discovery—the 3-angulated nature of the derived category $D(\mathcal{T}\text{-Mod})$—provides a definitive answer to the question of ternary periodicity. While binary structures are governed by triangles, ternary systems naturally produce quadrilaterals, a shift that is reflected in the 3-ary long exact sequences derived in Section 4. This higher-order symmetry suggests that $\mathcal{T}\text{-Mod}$ is not merely an extension of binary algebra, but a distinct geometric species.

\begin{ack}
The author gratefully acknowledges the support and encouragement provided by the
Commissionerate of Collegiate Education (CCE), Government of Andhra Pradesh, and the
Principal, Government College (Autonomous), Rajahmundry.
\end{ack}

\begin{funding}
This work was not funded by any organisation 
\end{funding}



\end{document}